\documentclass[a4paper,10.5pt]{amsart}
\usepackage[matrix,arrow,tips,curve]{xy}
\usepackage[english]{babel}
\usepackage{amsmath}
\usepackage{amssymb}
\usepackage{tikz}
\usepackage{mathrsfs}
\usepackage{enumerate}
\usepackage{graphicx}
\usepackage{hyperref}
\usetikzlibrary{trees}
\usetikzlibrary{arrows}

\newtheorem{thm}{Theorem}[section]
\newtheorem{Lemma}[thm]{Lemma}
\newtheorem{Proposition}[thm]{Proposition}
\newtheorem{Corollary}[thm]{Corollary}

\newtheorem*{thm*}{Theorem}

\theoremstyle{definition}

\newtheorem{Problem}[thm]{Problem}

\newtheorem{Definition}[thm]{Definition}
\newtheorem{Remark}[thm]{Remark}

\newtheorem{Example}[thm]{Example}

\definecolor{wwwwww}{rgb}{0.4,0.4,0.4}

\newcommand{\Oh}{\mathscr{O}_{\mathbb{C}^n,0}} 
\renewcommand{\P}{\mathbb{P}}

\newcommand{\C}{\mathbb{C}}

\newcommand{\M}{\mathfrak{m}}

\DeclareMathOperator{\mult}{mult}

\hypersetup{pdfpagemode=UseNone}
\hypersetup{pdfstartview=FitH}
		
\begin{document}

\title{Reconstruction of a Hypersurface Singularity from its Moduli Algebra}

\author[João Hélder Olmedo Rodrigues]{Jo\~ao H\'elder Olmedo Rodrigues}
\address{\sc Jo\~ao H\'elder Olmedo Rodrigues\\
Instituto de Matem\'atica e Estat\'istica, Universidade Federal Fluminense, Rua Prof. Marcos Waldemar de Freitas Reis, Campus Gragoat\'a, Bloco G - S\~ao Domingos\\
24210-201, Niter\'oi, RJ\\ Brazil}
\email{joaohelder@id.uff.br}

\date{\today}

\thanks{Research partially supported by FAPERJ, ARC E-26/211.361/2019}

\begin{abstract}
In this paper we present a constructive method to characterize ideals of the local ring $\Oh$ of germs of holomorphic functions at $0\in\C^n$ which arise as the moduli ideal $\langle f,\M\, j(f)\rangle$, for some $f\in\M\subset\Oh$. A consequence of our characterization is an effective solution to a problem dating back to the 1980's, called the Reconstruction Problem of the hypersurface singularity from its moduli algebra. Our results work regardless of whether the hypersurface singularity is isolated or not. 
\end{abstract}

\maketitle
\tableofcontents

\section{Introduction} 

The main motivation for this research is a problem originated from a celebrated result of J. Mather and S. Yau \cite{MY} in the early eighties, which relates certain isomorphism classes of commutative $\C$-algebras to biholomorphic classes of isolated hypersurface singularities. Few years later, the main Theorem of \cite{MY} was generalized by T. Gaffney and H. Hauser \cite{GH} to the case of non-isolated hypersurface singularities.  These results, as we will explain below, throw important light into the problem of (biholomorphic) classification of hypersurface singularities.

We recall that a germ of a complex hypersurface $(X,0)\subset (\C^n,0)$ at the origin $0\in\C^n$ is defined as the zero set of some - non-trivial - principal ideal $I_X$ of the local ring $\Oh$, the ring of germs of holomorphic functions at $0\in\C^n$. A generator of $I_X$ - which is well-defined modulo multiplication by an invertible element in $\Oh$ -  is said to be an \emph{equation} for $(X,0)$; if $I_X=\langle f\rangle $ we often say that the germ \emph{is defined by} $f$. When convenient, we emphasize this fact writing $(X_f,0)\subset (\C^n,0)$. 

We say that two germs of hypersurfaces $(X_f,0)\subset (\C^n,0)$ and $(X_g,0)\subset (\C^n,0)$ are \emph{biholomorphically equivalent} if there exist small open neighbourhoods $U$ and $V$ of the origin $0\in\C^n$, where $f,\,g$ converge and a (germ of) biholomorphism $\phi:U\rightarrow V$ - which sends the origin to itself - such that $\phi(X_f\cap U)=X_g\cap V$. Put more algebraically, it is easy to verify that this holds if and only if there exists 
an invertible element $u\in\Oh$ such that $ug=\phi^*(f)$. It is often said in this case that the function germs $f,g$ are \emph{contact equivalent}, because they lie in the same orbit of the action of the \emph{contact group} $\mathcal{K}$ on $\Oh$ (see the book \cite{GLS} for definitions). The totality of germs of hypersurfaces all biholomorphically equivalent to one another is said to be a \emph{biholomorphic class}. One of the most famous numerical invariants of a biholomorphic class is the \emph{multiplicity} of its elements. Let $\M$ denote the unique maximal ideal of $\Oh$. We recall that if a system $\textbf{\underline{x}}=x_1,\ldots,x_n$ of generators of $\M\subset \Oh$ is chosen and if $(X_f,0)\subset (\C^n,0)$ is a germ of hypersurface, its multiplicity $\mult(X_f,0)=\mult(f)$ is the smallest degree $m$ of a non-zero homogeneous polynomial appearing in a series expansion $f(\textbf{\underline{x}})=f_m(\textbf{\underline{x}})+f_{m+1}(\textbf{\underline{x}})+\ldots\,\,$ Clearly the multiplicity of a germ doesn't depend on the choice of $\textbf{\underline{x}}=x_1,\ldots,x_n$. So, two biholomorphically equivalent germs of hypersurfaces in $(\C^n,0)$ have the same multiplicity. Another important numerical invariant of a biholomorphic class is the \emph{Tjurina number} of its elements. We recall that the Tjurina number of $(X_f,0)\subset (\C^n,0)$ is defined to be the complex vector space dimension - whenever it is finite - $\tau(X_f)$ of the \emph{Tjurina algebra} of $(X_f,0)\subset (\C^n,0)$, which is defined as the quotient algebra \[ A(X_f) = \Oh/\langle f,j(f)\rangle,\] where $j(f)$ is the ideal generated by the partial derivatives of $f$.  It turns out that $\tau(X_f)$ is finite if and only if $(X_f,0)\subset (\C^n,0)$ has an \emph{isolated singularity} at $0\in\C^n$. Notice that having same multiplicity and same Tjurina number are necessary - but certainly not sufficient - conditions for $(X_f,0)\subset (\C^n,0)$ and $(X_g,0)\subset (\C^n,0)$ to belong to the same biholomorphic class. In fact, the search for a reasonable set of invariants separating biholomorphic classes of hypersurface singularities is an open problem.

Now we introduce the notation necessary to explain how Mather-Yau and Gaffney-Hauser results contribute to this open problem: let $(X_f,0)\subset (\C^n,0)$ be a germ of hypersurface. 
We define the \emph{moduli algebra} of $(X_f,0)$ - or more accurately of $f$ - as the quotient ring \[ B(X_f) = B(f) = \Oh/\langle f,\M\, j(f)\rangle.\] The ideal $\langle f,\M\,j(f)\rangle$ appearing as the denominator will be called the \emph{moduli ideal} of $(X_f,0)\subset (\C^n,0)$ and we denote it as $T_{\mathcal{K}}(f)$. If $g$ is another generator of $I_{X_f}$ it is easy to check that $T_{\mathcal{K}}(f)=T_{\mathcal{K}}(g)$ and this shows that $B(X_f)$ really doesn't depend on the chosen generator for $I_{X_f}$. More generally, we observe that when $(X_f,0)$ and $(X_g,0)$ are biholomorphically equivalent germs of hypersurfaces, then from a relation of type $ug=\phi^*(f)$ as above, it is straightforward to check that the moduli algebras $B(X_f)$ and $B(X_g)$ are \emph{isomorphic} as $\C$-algebras. The converse holds but it is much more subtle, being the essential part of the aforementioned results of Mather-Yau and Gaffney-Hauser, which we state here:

\begin{thm*}[Mather, Yau;\, Gaffney, Hauser]\label{MY}
Let $(X_f,0)\subset (\C^n,0)$ and $(X_g,0)\subset (\C^n,0)$ denote two germs of complex hypersurfaces. The statements are equivalent:
\begin{enumerate}
\item $(X_f,0)\subset (\C^n,0)$ and $(X_g,0)\subset (\C^n,0)$ are biholomorphically equivalent;
\item $B(X_f)$ and $B(X_g)$ are isomorphic as $\C$-algebras.
\end{enumerate}
\end{thm*}

\smallskip

\begin{Remark} A few comments are in order:
\begin{itemize}
\item The original statement of Mather-Yau theorem (cf. \cite{MY}) in the case of isolated hypersurface singularities says that (1) and (2) are also equivalent to ``\emph{$A(X_f)$ and $A(X_g)$ are isomorphic as $\C$-algebras}". In the general case (non-isolated hypersurface singularities) this is not true any more, as shown by a counterexample constructed by Gaffney and Hauser (\cite{GH}).
\item In \cite{GH}, the authors show, by means of the introduction of a quotient module which plays the role of the moduli algebra $B(X_f)$, that a similar assertion holds beyond the case of hypersurface singularities.
\item Diverging from the terminology adopted here, in \cite{MY} the authors originally baptised $A(X_f)$ and $B(X_f)$ as the \emph{moduli algebras} of the germ of hypersurface $(X_f,0)\subset (\C^n,0)$, because the preceding result tells us that the problem of the classification of germs of (isolated) hypersurfaces singularities $(X_f,0)\subset (\C^n,0)$ up to biholomorphic equivalence is equivalent to that of the classification of the algebras $A(X_f)$ or $B(X_f)$ up to $\C$-algebra isomorphism.
\item The proofs presented in \cite{MY} and \cite{GH} are not constructive. Until very recently there was the open problem, called the \emph{Reconstruction Problem}, of reconstructing the (isolated) hypersurface singularity out of its Tjurina algebra. For solutions in special cases we refer to \cite{Y}, \cite{IK}, \cite{E} and the related work \cite{ES}. In \cite{OR} we solved the Reconstruction Problem, at least in the case where the hypersurface can be characterized by its Tjurina algebra. This is precisely the case where the hypersurface singularity is of \emph{Isolated Singularity Type} (cf. \cite{GH} for details). This case includes - strictly - the case of isolated hypersurface singularities. The main purpose of this paper is to show that we can push our techniques a little further and reconstruct the hypersurface out of $B(X_f)$, then closing the remaining cases.
\item Related to the Reconstruction problem is the \emph{Recognition problem}, which is to decide whether a quotient algebra $\Oh/I$ is isomorphic to the Tjurina algebra of some hypersurface singularity $(X_f,0)\subset (\C^n,0)$. This is of course equivalent to recognize whether the ideal $I\subset \Oh$ is a Tjurina ideal $\langle f,j(f)\rangle$ for some $f\in\M\subset\Oh$. This was the approach taken in \cite{OR}.
\end{itemize}
\end{Remark}

To deal with the problem of reconstruction described above for algebras of type $B(X_f)$ we introduce right away our guiding question through this paper, namely

\begin{Problem}\label{pergunta} Fixed $n\geqslant 1$, how to find necessary and sufficient conditions for a (proper) ideal $I\,\subset\Oh$ to be the moduli ideal of some $f\in\M$?
\end{Problem}

Clearly, for all $n\geqslant 1$, the zero ideal and the maximal ideal $\M$ are the moduli ideals of $0$ and $x_1$, say, respectively. 

\begin{Example} If $n=1$ then any ideal $I\subset \mathscr{O}_{\C^1,0}=\C\{x_1\}$ is of the form $I=\M^k=\langle x_1^k\rangle$, $k\geqslant 1$; this is the moduli ideal of $f(x_1)=x_1^{k}$.
\end{Example}

So, if $n=1$, Problem \ref{pergunta} has a trivial solution. However, beginning at $n=2$, this is not true any more. Clearly a necessary condition for an ideal $I\,\subset\Oh$ to be the moduli ideal of some $f\in \Oh$ is that the minimal number of generators of $I$ should be at most $n^2+1$, but this is by no means sufficient as we see in the next example. 

\begin{Example}\label{ex1} For any $n\geqslant 2$ the ideal $I_n=\langle x_1^2,x_2,x_3,\ldots,x_n\rangle\subset\Oh$ is not a moduli ideal. Indeed, since $\dim_{\C}\,\Big(\frac{\Oh}{I_n}\Big)=2$ the Tjurina number of a possible $f$ satisfying $I_n=T_{\mathcal{K}}(f)$ is at most $2$. Up to contact equivalence there are finitely many $f$ such that $\tau(X_f)\leqslant 2$, namely $f_0=x_1$, $f_1=x_1^2+x_2^2+\ldots+x_n^2$ and $f_2=x_1^3+x_2^2+\ldots+x_n^2$. By direct inspection one checks that the corresponding quotients $\frac{\Oh}{T_{\mathcal{K}}(f_i)}$ have $\C$-vector space dimensions $1$, $n+1$ and $n+2$ respectively. Since $n\geqslant 2$, none has $\C$-vector space dimension $2$. \end{Example}

\begin{Remark} The previous example serves as an illustration of the fact that being a moduli algebra is a property that ``depends on the embedding". Indeed, $\Oh/I_n\simeq \mathscr{O}_{\C^1,0}/\langle x_1^2\rangle$ as abstract (or non-embedded) $\C$-algebras but $\langle x_1^2\rangle\subset\mathscr{O}_{\C^1,0}$ is a moduli ideal, while if $n\geqslant 2$, $I_n\subset\Oh$ is not.
\end{Remark}

\begin{Example}\label{guia}Consider the family of powers of the maximal ideal, $\M^k\subset\mathscr{O}_{\C^2,0}=\C\{x,y\}$. If $k\geqslant 5$ no such ideal is a moduli ideal, because they are minimally generated by more than $5=2^2+1$ elements. More interesting are the cases $k\leqslant 4$. If $k\leqslant 3$, we check easily that $\M$, $\M^2$, $\M^3$ are moduli ideals of $x,\,xy$ and $x^3+y^3$ respectively, but notice that these function germs are choices. Indeed, a ``sufficiently general" element of $\M$, $\M^2$ and $\M^3$ respectively, is a function germ of which they are moduli ideals. This trick of taking a sufficiently general element in the ideal under consideration does not work for $\M^4$: we claim that $\M^4$ is not a moduli ideal of any $f\in\mathscr{O}_{\C^2,0}$.

For, since $\dim_{\C}\,\Big(\frac{\mathscr{O}_{\C^2,0}}{\M^4}\Big)=10$, the Tjurina number $\tau=\tau(X_f)$ of a possible $f\in\mathscr{O}_{\C^2,0}$ such that $T_{\mathcal{K}}(f)=\M^4$ is at most $10$. After a check to the Arnold's lists in \cite{Ar} we see that any $f$ with $\tau\leqslant 10$ is quasi-homogeneous; hence $f\in\M\,j(f)\subset j(f)$. It follows from the natural exact sequence \[0\rightarrow j(f)/\M\,j(f)\rightarrow \mathscr{O}_{\C^2,0}/\M\,j(f)\rightarrow \mathscr{O}_{\C^2,0}/j(f)\rightarrow 0,\] that $\dim_{\C}\,\Big(\frac{\mathscr{O}_{\C^2,0}}{T_{\mathcal{K}}(f)}\Big)=\tau+2$. We deduce that the possible $f$ must define $A_8$, $D_8$ or $E_8$ singularities. Up to contact equivalence we can compute with the respective normal forms $x^2+y^9$, $x^2y+y^7$ and $x^3+y^5$ to obtain $\langle f,\M\, j(f)\rangle$ as $\langle x^2,xy,y^9\rangle$, $\langle x^3,x^2y,xy^2,y^7\rangle$ and $\langle x^3,x^2y,xy^4,y^5\rangle$ respectively. All of them have elements of multiplicity smaller than $4$, opposite to $\M^4$. Since multiplicity is invariant under contact equivalence, we conclude that $\M^4$ is not a moduli ideal.
\end{Example}

The above example already contains the main strategy through the paper, which is to split the Problem \ref{pergunta} in two parts. The first part is to answer, given an ideal $I\subset\Oh$, \emph{where to?} search for a solution to the equation $I=T_{\mathcal{K}}(f)$. Trying to answer this question will lead us to the notion of $\Delta_1(I)$, the set of anti-derivatives $I$. This is the natural place inside $\Oh$ to look at in the search for a solution $f$. We treat this in Section \ref{basics}, in which we will also discuss some preliminaries on moduli ideals. In Section \ref{computador} we suggest an easily applicable method for computation of $\Delta_1(I)$ in examples, with routines already implemented in SINGULAR \cite{DGPS}. The second part of our strategy is an effort to give a precise meaning to the somewhat vague term \emph{sufficiently general} used in the previous example. This is carried out in Section \ref{segunda} where we introduce two easily checkable properties on arbitrary ideals of $\Oh$. Our main result, to be presented in Section \ref{teoremas}, is a characterization of the ideals $I$ for which the equation $I=T_{\mathcal{K}}(f)$ admits a solution $f$. In other words, we characterize moduli ideals, by means of which we give an explicit solution for Problem \ref{pergunta}. In the same section we show, in examples, how to use SINGULAR to recognize a moduli ideal and to reconstruct the hypersurface singularity from it.

\subsection*{Acknowledgments}
The author wishes to express his gratitude to T. Gaffney for patiently answering questions related to Mather-Yau type results; and to G.-M. Greuel for bringing to his attention an inaccuracy written in a previous preprint.

\section{Preliminaries}\label{basics}
Let $\Oh$ the local ring of germs of holomorphic functions at $0\in\C^n$ and let $\M$ denote its maximal ideal. For some germ of holomorphic function $f\in\M$ we let $j(f)$ denote the \emph{Jacobian ideal} of $f$, that is, the ideal generated by the (first order) partial derivatives of $f$ with respect to a chosen coordinate system $\textbf{\underline{x}}=x_1,\ldots,x_n$ - minimal set of generators - for $\M$. As indicated in the Introduction, we will be concerned with moduli ideals \[T_{\mathcal{K}}(f)=\langle f,\M j(f)\rangle\] of function germs $f\in\M$. It is easy to verify that for any biholomorphic change of coordinates $\phi:(\C^n,0)\rightarrow (\C^n,0)$ and any $f\in\Oh$, we have $T_{\mathcal{K}}(\phi^*(f))=\phi^*(T_{\mathcal{K}}(f))$. Hence we will fix a coordinate system and we will always compute $T_{\mathcal{K}}(f)$ with respect to this coordinate system. 

The properties below of $T_{\mathcal{K}}$ are immediate to verify:

\begin{Remark}\label{aritmetica} For any $f,g\in\Oh$ we have:
\item (i) $T_{\mathcal{K}}(f+g)\subseteq T_{\mathcal{K}}(f)+T_{\mathcal{K}}(g)$;
\item (ii) $T_{\mathcal{K}}(fg)\subseteq T_{\mathcal{K}}(f)\, T_{\mathcal{K}}(g)$;
\item (iii) $T_{\mathcal{K}}(f)=T_{\mathcal{K}}(g)$ if $\langle f\rangle=\langle g\rangle$.
\end{Remark}

We will find it convenient to introduce the definition below:

\begin{Definition}\label{tdeummodulo}
Let $J\subset \Oh$ the ideal generated by $g_1,\ldots, g_q\in\Oh$. We define the ideal \[T_{\mathcal{K}}(J)=T_{\mathcal{K}}(g_1)+\ldots+T_{\mathcal{K}}(g_q)\subseteq \Oh.\]
\end{Definition}

Notice that this is a well-posed definition since $T_{\mathcal{K}}(J)$ does not depend on the generators $g_i$ chosen but only on the ideal they generate, as the reader can check using $(i),\,(ii)$ in Remark \ref{aritmetica}. Although Problem \ref{pergunta} only mentions moduli ideals of function germs (or rather of principal ideals, cf.(iii) of Remark \ref{aritmetica}), it seems fair to refer to $T_{\mathcal{K}}(J)$ as the \emph{moduli ideal} of $J$.

\begin{Remark} The properties below are easy to verify.
\begin{itemize}
\item[(i)] $J\subseteq T_{\mathcal{K}}(J)$;
\item[(ii)] If $J_1\subseteq J_2$ is an inclusion of ideals then $T_{\mathcal{K}}(J_1)\subseteq T_{\mathcal{K}}(J_2)$;
\item[(iii)] If $\{J_{\lambda}\}_{\lambda\in\Lambda}$ is any family of ideals of $\Oh$ then $T_{\mathcal{K}}(\sum_{\lambda}J_{\lambda})=\sum_{\lambda}T_{\mathcal{K}}(J_{\lambda})$;
\item[(iv)] $T_{\mathcal{K}}(J_1\cap J_2)\subseteq T_{\mathcal{K}}(J_1)\cap T_{\mathcal{K}}(J_2)$.
\end{itemize}
\end{Remark}

Now we introduce the main object - which is an adapted version of the object used in \cite{OR} and - that will ultimately lead us to our solution to Problem \ref{pergunta}.

\begin{Definition}\label{delta}
Let $I\subset\Oh$ be an ideal. We define the \emph{ideal of anti-derivatives} of $I$ as \[\Delta_1(I)=\{f\in\Oh\,|\,T_{\mathcal{K}}(f)\subseteq I\}.\]
\end{Definition}

It is easy to show that $\Delta_1(I)$ does not depend on the choice of parameters $\textbf{\underline{x}}=x_1,\ldots,x_n$ used to compute Jacobian ideals. Moreover, the properties below are straightforward to check.

\begin{Remark}\label{easy}
\item[(i)] $\Delta_1(I)$ is an ideal of $\Oh$;
\item[(ii)] $I^2\subseteq \Delta_1(I)\subseteq I$;
\item[(iii)] If $\{I_{\lambda}\}_{\lambda\in\Lambda}$ is any family of ideals of $\Oh$ then $\Delta_1(\bigcap_{\lambda}I_{\lambda})=\bigcap_{\lambda}\Delta_1(I_{\lambda})$;
\item[(iv)] If $I\subseteq J$ is an inclusion of ideals in $\Oh$ then $\Delta_1(I)\subseteq \Delta_1(J)$;
\end{Remark}


\begin{Example}\label{principal1}Let $I=\langle f^k \rangle\subset\Oh$, $n\geqslant 2$, be a principal ideal with $k\geqslant 1$ and $f$ being an irreducible element. Then $\Delta_1(I)=\langle f^{k+1}\rangle$. Indeed, one inclusion is immediate. For the opposite inclusion, assume $g\in \Delta_1(I)$ and write $g=af^k$, for some $a\in\Oh$. For all $i,j=1,\ldots,n$, $x_i\frac{\partial g}{\partial x_j}=x_if^k \frac{\partial a}{\partial x_j}+akf^{k-1}x_i\frac{\partial f}{\partial x_j}$ is also a multiple of $f^k$. Hence $f$ must divide $ax_i\frac{\partial f}{\partial x_j}$, for all $i,j=1,\ldots,n$. Since $f$ is irreducible, for every $i,j=1,\ldots,n$, $f$ must divide $x_i$ or $\frac{\partial f}{\partial x_j}$ or $a$. If $f$ divides $\frac{\partial f}{\partial x_j}$ for all $j=1,\ldots,n$ then $\langle \frac{\partial f}{\partial x_1},\ldots, \frac{\partial f}{\partial x_n}\rangle\subseteq \langle f \rangle$ and we would deduce that $f=0\in\Oh$, which is not the case. Since $n\geqslant 2$, $f$ cannot divide all the $x_i$. So, for some $i,j$, $f$ does not divide $x_i\frac{\partial f}{\partial x_j}$. Being an irreducible element, $f$ must divide $a$. Hence $g\in\langle f^{k+1}\rangle$.\end{Example}

\begin{Example}\label{principal2} Let $I=\langle f\rangle \subset \Oh$, $n\geqslant 2$, be a non-trivial principal ideal and $f=f_1^{k_1}\ldots f_r^{k_r}$ be a factorization of $f$ into irreducible, non-associated, elements with positive $k_1,\ldots,k_r$. Then $I=\langle  f_1^{k_1}\rangle \cap\ldots\cap \langle  f_r^{k_r}\rangle $. It follows from Remark \ref{easy}, (iii) and Example \ref{principal1} that $\Delta_1(I)=\langle f_1^{k_1+1}\rangle \cap\ldots\cap \langle  f_r^{k_r+1}\rangle =\langle  f_1^{k_1+1}\ldots f_r^{k_r+1}\rangle $.
\end{Example}

\begin{Example}\label{potenciadem} Let $\M^k$, with $k\geqslant 1$, denote the $k$-th power of the maximal ideal of $\Oh$. Let $g=x_1^{k_1}\ldots x_n^{k_n}$ with $k_j\geqslant 0$ and $\sum k_j=k$ be a monomial generator of $\M^k$. For any pair $i,j=1,\ldots,n$ it is easy to check that $x_i\frac{\partial g}{\partial x_j}\in\M^k$. Hence $T_{\mathcal{K}}(g)\subset \M^k$, which shows $\Delta_1(\M^k)=\M^k$.\end{Example}

\begin{Example}\label{outrodelta}Let $I_n=\langle x_1^2,x_2,\ldots,x_n\rangle\subset\Oh$, $n\geqslant 2$, as in Example \ref{ex1}. Then $\Delta_1(I_n)=\M^2$. Indeed, since $\M^2\subset I_n$ then according to Remark \ref{easy}, item (iv) and the previous example we have $\M^2=\Delta_1(\M^2)\subseteq\Delta_1(I_n)$. For the other inclusion, if $g=a_1x_1^2+a_2x_2+\ldots+a_nx_n\in \Delta_1(I_n)$, computing $\frac{\partial g}{\partial x_j}$, we see that $a_j\in (I_n:\M)=\M$, for all $j\geqslant 2$. Hence $a_2x_2+\ldots+a_nx_n\in\M^2$ and we conclude that $g\in \M^2$.
\end{Example}

\section{Computation of $\Delta_1$}\label{computador}

Up to now we have obtained the ideal of anti-derivatives $\Delta_1(I)$ in a few specific cases (cf. Examples \ref{principal2}, \ref{potenciadem}, \ref{outrodelta}). In this section we suggest a method to compute $\Delta_1(I)$ for any ideal $I\subset \Oh$, then showing that it is accessible in concrete computations. We will illustrate our procedure with examples obtained using basic routines already implemented in the software SINGULAR, \cite{DGPS}.

Fix a basis $\textbf{\underline{x}}=x_1,\ldots,x_n$ of the maximal ideal $\M\subset\Oh=\C\{x_1,\ldots,x_n\}$. A free $\Oh$-module of certain rank $\ell$ will be denoted by $F_{\ell}$; for any $\underline{b}=(b_1,\ldots,b_{\ell})^{t}$, $\underline{c}=(c_1,\ldots,c_{\ell})^t\in F_{\ell}$ we use a dot notation $\underline{b}\cdot\underline{c}$ to represent $\sum_{k=1}^{\ell}b_k c_k\in \Oh$.

Assume the ideal $I$ given by generators, $I=\langle f_1,\ldots,f_{\ell}\rangle $. Then some element $g=\underline{a}\cdot\underline{f} = \sum_k a_kf_k\in I$ is an element of $\Delta_1(I)$ if and only if, for all $i,j=1,\ldots,n$, $$x_i\frac{\partial g}{\partial x_j}=\sum_k a_k x_i\frac{\partial f_k}{\partial x_j}+\sum_k \frac{\partial a_k}{\partial x_j}x_if_k=\underline{a}\cdot x_i\frac{\partial \underline{f}}{\partial x_j}+\frac{\partial \underline{a}}{\partial x_j}\cdot x_i\underline{f}$$ belongs to $I$. In other words, $g=\underline{a}\cdot\underline{f}\in \Delta_1(I)$ if and only if $\underline{a}\cdot x_i \frac{\partial \underline{f}}{\partial x_j}\in I,$ for all $i,j=1,\ldots,n$.

For each $i,j=1,\ldots,n$ we denote by $E^{i,j}$ the submodule of $F_{\ell}$ consisting of the elements $\underline{a}\in F_{\ell}$ such that $\underline{a}\cdot x_i \frac{\partial \underline{f}}{\partial x_j}\in I$. Then $E:=\bigcap_{i,j=1}^n\,E^{i,j}\subseteq F_{\ell}$ is a finitely generated $\Oh$-submodule of $F_{\ell}$ consisting of all elements $\underline{a}\in\,F_{\ell}$ such that $g=\underline{a}\cdot\underline{f}\in\Delta_1(I)$. \smallskip

Notation being as above we have

\begin{Proposition}\label{compdelta} Let $\underline{e}_{\,1},\ldots,\underline{e}_{\,q}\in F_{\ell}$ be generators of $E$. Then the ideal of anti-derivatives $\Delta_1(I)$ is generated by $\underline{e}_{\,t}\cdot\underline{f}$, for $t=1,\ldots,q$.
\end{Proposition}

\begin{proof}
As already discussed, $g\in \Delta_1(I)$ if and only if $g=\underline{a}\cdot\underline{f}$ for some $\underline{a}\in E$. Since $E$ is generated by all the $\underline{e}_{\,t}$, the result follows from $\Oh$-linearity of the $(-)\cdot\underline{f}$ product.
\end{proof} 

Here we show how we used the software SINGULAR, \cite{DGPS} to compute the ideal of anti-derivatives of ideals to be presented in several of our examples. We claim no originality, since only routines already implemented by the software developers and collaborators were applied.

\begin{Example}\label{quatrovariaveis}Let $I=\langle yz,z^3,xw,w^2\rangle\subset \mathscr{O}_{\C^4,0}=\C\{x,y,z,w\}$. With the interface of SINGULAR open, type

\begin{flushleft}
\texttt{> ring r=0,(x,y,z,w),ds;}
\end{flushleft}

\noindent This declares you are working over a field of characteristic zero, variables $x,y,z,w$ and set the corresponding ring of power series. We now declare the generators of the ideal $I$ by means of a matrix with one row and (in the present case) four columns: type \begin{flushleft}
\texttt{> matrix B[1][4]=yz,z3,xw,w2;}
\end{flushleft}

\noindent Now to compute the submodule $E^{1,1}$, (same notation as above), we declare a matrix with entries the partial derivatives of the given generators in terms of $x$ multiplied by $x$. We compute $E^{1,1}$ as follows: 

\begin{flushleft}
 \texttt{> matrix e11[1][4]=x*diff(B,x); def E11=modulo(e11,B)};
 \end{flushleft}
 
\noindent Likewise, compute $E^{1,2}$, $E^{1,3}$, $E^{1,4}$, $E^{2,1}$, $E^{2,2}$, $E^{2,3}$, $E^{2,4}$, $E^{3,1}$, $E^{3,2}$, $E^{3,3}$, $E^{3,4}$, $E^{4,1}$, $E^{4,2}$, $E^{4,3}$, $E^{4,4}$:

\begin{flushleft}
\begin{flushleft}
 \texttt{> matrix e12[1][4]=x*diff(B,y); def E12=modulo(e12,B)};
 \end{flushleft}
 \begin{flushleft}
 \texttt{> matrix e13[1][4]=x*diff(B,z); def E13=modulo(e13,B)};
 \end{flushleft}
 \vdots
\begin{flushleft}
 \texttt{> matrix e44[1][4]=w*diff(B,w); def E44=modulo(e44,B)};
 \end{flushleft} 
 
\end{flushleft}

\noindent Now, we put

\begin{flushleft}
\texttt{> def e=intersect(E11,E12,E13,E14,E21,E22,E23,E24,E31,E32,E33,E34,E41,E42,E43,E44);}\\
\texttt{> def E=std(e);}
\end{flushleft}

\noindent This defines $E$ and computes a standard basis, with respect to the given monomial order. Now, to obtain $E$, type
\begin{flushleft}
\texttt{> print(E);}
\end{flushleft}

\noindent SINGULAR gives

\medskip

\noindent\makebox[\textwidth]{%
\small
\(
\left( \begin{array}{ccccccccccccccc}
0 & 0 & 0 & 0 & 0 & 0 & 0  & 0  & yz & 0   & z^2 & xw & zw & w^2 & 0\\
0 & y & z & 0 & 0 & w & 0  & 0  & 0  & 0   & 0   & 0  & 0  & 0   & 0\\
0 & 0 & 0 & 0 & w & 0 & 0  & yz & 0  & 0   & 0   & 0  & 0  & 0   & z^3\\
x & 0 & 0 & w & 0 & 0 & yz & 0  & 0  & z^2 & 0   & 0  & 0  & 0   & 0\\
\end{array} \right)\)}

\medskip

\noindent Hence, as in Proposition \ref{compdelta}, we obtain generators of $\Delta_1(I)$ computing \[(-)\cdot(yz,z^3,xw,w^2),\] for all (transposed) columns of the previous matrix. We speed up calculations typing 

\begin{flushleft}
\texttt{> ideal d=B*E;}\\
\texttt{> ideal D=std(d);}\\
\texttt{> D;}
\end{flushleft}

The output lists the generators of $\Delta_1(I)$. In the case under consideration we obtain the ten generated monomial ideal below: \[\Delta_1(I)=\langle xw^2,w^3,y^2z^2,yz^3,z^4,xyzw,yz^2w,z^3w,yzw^2,z^2w^2\rangle.\]
\end{Example}

\begin{Example}\label{tresvariaveis}Let $I=\langle x^3,x^2y,x^2z,xy^3z,y^4z,xy^2z^2,y^3z^2,y^2z^3\rangle \subset \C\{x,y,z\}$.

\smallskip

With the interface of SINGULAR open, type

\begin{flushleft}
\texttt{> ring r=0,(x,y,z),ds;}\\
\texttt{> matrix B[1][8]=x3,x2y,x2z,xy3z,y4z,xy2z2,y3z2,y2z3;}
\end{flushleft}
\noindent Compute $E^{1,1}$,$E^{1,2}$, $E^{1,3}$, $E^{2,1}$, $E^{2,2}$, $E^{2,3}$, $E^{3,1}$, $E^{3,2}$, $E^{3,3}$:
\begin{flushleft}
 \texttt{> matrix e11[1][8]=x*diff(B,x); def E11=modulo(e11,B)};
\end{flushleft}
\begin{flushleft}
\begin{flushleft}
 \texttt{> matrix e12[1][8]=x*diff(B,y); def E12=modulo(e12,B)};
 \end{flushleft}
 \vdots
\begin{flushleft}
 \texttt{> matrix e33[1][8]=z*diff(B,z); def E33=modulo(e33,B)};
 \end{flushleft} 
\end{flushleft}
\begin{flushleft}
\texttt{> def e=intersect(E11,E12,E13,E21,E22,E23,E31,E32,E33);}\\
\texttt{> def E=std(e);}\\
\texttt{> print(E);}
\end{flushleft}

\noindent This time, SINGULAR gives a $8\times 14$ matrix. Again we obtain generators of $\Delta_1(I)$ computing \[(-)\cdot(x^3,x^2y,x^2z,xy^3z,y^4z,xy^2z^2,y^3z^2,y^2z^3),\] for all (transposed) columns of the matrix mentioned.

\begin{flushleft}
\texttt{> ideal d=B*E;}\\
\texttt{> ideal D=std(d);}\\
\texttt{> D;}
\end{flushleft}

The output lists the generators of $\Delta_1(I)$. In the case under consideration we obtain the ideal whose generators are listed below:

\begin{flushleft}
\texttt{D[1]=x3}\\
\texttt{D[2]=x2y2z}\\
\texttt{D[3]=y3z2}\\
\end{flushleft}

\end{Example}

\medskip

\begin{Example} \label{duasvariaveis}
Let $I=\langle 3xy^2+x^6,y^3,x^5y,x^7\rangle\subset \mathscr{O}_{\C^2,0}=\C\{x,y\}$. We repeat a procedure analogous as those in previous examples.
\begin{flushleft}
\texttt{> ring r=0,(x,y),ds;}
\end{flushleft}
\begin{flushleft}
\texttt{> matrix B[1][4]=3xy2+x6,y3,x5y,x7;}
\end{flushleft}
\begin{flushleft}
\begin{flushleft}
\texttt{> matrix e11[1][4]=x*diff(B,x); def E11=modulo(e11,B)};
 \end{flushleft}
\begin{flushleft}
 \texttt{> matrix e12[1][4]=x*diff(B,y); def E12=modulo(e12,B)};
 \end{flushleft}
 \begin{flushleft}
 \texttt{> matrix e21[1][4]=y*diff(B,x); def E21=modulo(e21,B)};
 \end{flushleft}
 \begin{flushleft}
 \texttt{> matrix e22[1][4]=y*diff(B,y); def E22=modulo(e22,B)};
 \end{flushleft} 
 \end{flushleft}
\begin{flushleft}
\texttt{> def e=intersect(E11,E12,E21,E22);}\\
\texttt{> def E=std(e);}\\
\texttt{> print(E);}
\end{flushleft}

\noindent In this case, SINGULAR gives

\medskip

\noindent\makebox[\textwidth]{%
\small
\(
\left( \begin{array}{cccccc}
0 & 0 & 0 & 0 & y & x^3 \\
0 & 1 & x & y & 0 & 0 \\
0 & 1 & 0 & 0 & 0 & 0 \\
1 & 0 & 0 & 0 & 0 & 0 \\
\end{array} \right)\)}

\medskip

\noindent Again, as in Proposition \ref{compdelta}, generators of $\Delta_1(I)$ are obtained computing \[(-)\cdot(3xy^2+x^6,y^3,x^5y,x^7),\] for all (transposed) columns of the previous matrix. We do this as 

\begin{flushleft}
\texttt{> ideal d=B*E;}\\
\texttt{> ideal D=std(d);}\\
\texttt{> D;}
\end{flushleft}

We obtain $\Delta_1(I)=\langle y^3+x^5y, x^4y^2,x^7,x^6y\rangle.$
\end{Example}

\section{Two properties of ideals in $\Oh$}\label{segunda}

In this section we introduce the two relevant conditions present in our characterization (see Theorem \ref{main}) of moduli ideals of hypersurface singularities in $\Oh$.

\subsection*{$T_{\mathcal{K}}$-fullness}

Here we introduce $T_{\mathcal{K}}$-fullness, a quite natural notion related to our Problem \ref{pergunta}. Recall the definition of the moduli ideal of an arbitrary ideal of $\Oh$ (cf. Definition  \ref{tdeummodulo}) and observe that, in general, we have $T_{\mathcal{K}}(\Delta_1(I))\subseteq I$.

\begin{Definition}\label{tfullness} Let $I$ be an ideal of $\Oh$. We say that $I$ is \emph{$T_{\mathcal{K}}$-full} if $T_{\mathcal{K}}(\Delta_1(I))=I$.\end{Definition}

\begin{Example}\label{naotfull} Let $I_n=\langle x_1^2,x_2,\ldots,x_n\rangle\subset\Oh$, $n\geqslant 2$ as in Example \ref{ex1}. We have seen in Example \ref{outrodelta} that $\Delta_1(I_n)=\M^2$. A routine calculation now shows that $T_{\mathcal{K}}(\Delta_1(I_n))=T_{\mathcal{K}}(\M^2)=\M^2\subsetneq I_n$ so $I_n$ is not $T_{\mathcal{K}}$-full.\end{Example}

The general significance of $T_{\mathcal{K}}$-fullness in our investigation is apparent in the next result, which is, in view of Example \ref{naotfull}, an alternative proof that $I_n\subset\Oh$, as in Example \ref{ex1}, is not a moduli ideal.

\begin{Proposition}\label{tfullehnecess} Let $I$ be an ideal of $\Oh$. If $I$ is a moduli ideal then $I$ is $T_{\mathcal{K}}$-full.\end{Proposition}

\begin{proof}
Let $\Delta_1(I)=\langle g_1,\ldots,g_q\rangle $. If $I$ is a moduli ideal, then there exists $f\in\Delta_1(I)$ such that $I=T_{\mathcal{K}}(f)$. We may write $f=\sum_{k=1}^q r_k g_k$, for some $r_k\in \Oh$. Using Remark \ref{aritmetica}, we obtain $I=T_{\mathcal{K}}(f)\subseteq \sum_{k=1}^q T_{\mathcal{K}}(r_k g_k)\subseteq \sum_{k=1}^q  T_{\mathcal{K}}(g_k)=T_{\mathcal{K}}(\Delta_1(I))\subseteq I$ and equality holds throughout. We conclude that $I$ is $T_{\mathcal{K}}$-full.
\end{proof}

\begin{Example}(Example \ref{quatrovariaveis}, continued) We have computed the anti-derivatives ideal of $$I=\langle yz,z^3,xw,w^2\rangle\subset \C\{x,y,z,w\}$$ as $\Delta_1(I)=\langle xw^2,w^3,y^2z^2,yz^3,z^4,xyzw,yz^2w,z^3w,yzw^2,z^2w^2\rangle$. We check easily that $T_{\mathcal{K}}(\Delta_1(I))\subsetneq I$. Hence $I$ is not a moduli ideal because it is not $T_{\mathcal{K}}$-full.\end{Example}

Being $T_{\mathcal{K}}$-full is, however, not sufficient for $I$ to be a moduli ideal.

\begin{Example}\label{cota} Let $k\geqslant 5$ and let $I=\M^k\subset\mathscr{O}_{\C^2,0}=\C\{x,y\}$. Certainly $I$ is not a moduli ideal (see Example \ref{guia}). However, we have seen (cf. Example \ref{potenciadem}) that $\Delta_1(I)=I$. It is equally easy to check that $T_{\mathcal{K}}(\Delta_1(I))=I$, which implies that $I$ is $T_{\mathcal{K}}$-full.\end{Example}

\subsection*{$T_{\mathcal{K}}$-dependence} 

We have shown above that $T_{\mathcal{K}}$-fullness is not sufficient for an ideal to be a moduli ideal. Here we explain the last ingredient needed, in addition to  $T_{\mathcal{K}}$-fullness (cf. Definition \ref{tfullness}), to characterize moduli ideals. 

We aim to give a precise meaning to the intuition hinted already in Example \ref{guia}, that of a general linear combination $f$ of the given generators of $\Delta_1(I)$ has the largest possible $T_{\mathcal{K}}(f)$ and could reveal whether a given ideal $I$ is a moduli ideal. Our results and definitions in this subsection are expressed geometrically, since it seemed to us more appropriate to explain these ideas. To this end, we use basic concepts on schemes, consistent with \cite{H}, Chapter II, to which we refer for terminology.

We regard ideals in $\Oh$ as ideal sheaves on the affine scheme $\mbox{Spec}\,\Oh$. Let $J\subset \Oh$ be any ideal and assume that $J$ is given by generators: $J=\langle g_1,\ldots,g_q\rangle \subset \Oh$; then we consider the projective $(q-1)$-space over $\Oh$, namely, $\P^{q-1}=\mbox{Proj}(S)$, where $S=\Oh[\alpha_1,\ldots,\alpha_{q}]=\bigoplus_{d\geqslant 0}\,S_d$ is the standard polynomial ring over $\Oh$, with variables $\alpha_i$, graded so that $\deg\alpha_i=1$, for all $i$.

Let $\pi:\P^{q-1}\rightarrow \mbox{Spec}\,\Oh$ be the natural morphism of $\mbox{Spec}\,\Oh$-schemes and let $\sigma$ denote the global section $\sum_{i=1}^q\,g_i\alpha_i$ of $\pi^*(J)\otimes \mathcal{O}_{\P^{q-1}}(1)$. Then there is a \emph{moduli ideal sheaf} of $\sigma$ on $\P^{q-1}$, namely $\mathscr{T}_{\mathcal{K}}(\sigma)$, the sheaf associated to the homogeneous ideal $\langle \sigma,\M\frac{\partial \sigma}{\partial x_1},\ldots,\M\frac{\partial \sigma}{\partial x_n}\rangle $ of $S$. Clearly $\mathscr{T}_{\mathcal{K}}(\sigma)$ is a subsheaf of $\pi^*(T_{\mathcal{K}}(J))$, so we can take the quotient sheaf $$\mathscr{F}=\frac{\pi^*(T_{\mathcal{K}}(J))}{\mathscr{T}_{\mathcal{K}}(\sigma)}$$ on $\P^{q-1}$. Since $\mathscr{F}$ is coherent and $\P^{q-1}$ is noetherian, the support $\mbox{Supp}\mathscr{F}$ of $\mathscr{F}$ is a closed subscheme of $\P^{q-1}$ given by the vanishing of $\Big(\mathscr{T}_{\mathcal{K}}(\sigma):\pi^*(T_{\mathcal{K}}(J))\Big)$.

\begin{Definition}\label{tdependence} We say that an ideal $J\subset\Oh$ is \emph{$T_{\mathcal{K}}$-dependent} if $\pi^{-1}(\M)\not\subset\mbox{Supp}\mathscr{F}$. 
\end{Definition}

We check that $T_{\mathcal{K}}$-dependence for $J$ is a well-defined concept, being independent on the choice of generators for $J$. For, let the ideal $J$ be given by another system of generators, say $J=\langle h_1,\ldots,h_u \rangle $. Let $\P^{u-1}=\mbox{Proj}(S')$, being $S'=\Oh[\,\beta_1,\ldots,\beta_{u}]$. The above construction can be carried out with the obvious morphism $\pi':\P^{u-1}\rightarrow \mbox{Spec}\,\Oh$ and section $\sigma'=\sum_{j=1}^u\,h_j\beta_j$ instead of $\pi$ and $\sigma$, obtaining a sheaf $\mathscr{F'}$ on $\P^{u-1}$.

\begin{Lemma}Keeping notation as above, we have $\pi^{-1}(\M)\not\subset\mbox{Supp}\mathscr{F}$ if and only if $\pi'^{-1}(\M)\not\subset\mbox{Supp}\mathscr{F'}$.\end{Lemma}

\begin{proof} We will only show that $\pi^{-1}(\M)\not\subset\mbox{Supp}\mathscr{F}$ implies $\pi'^{-1}(\M)\not\subset\mbox{Supp}\mathscr{F'}$. The proof of the other implication is analogous. 

Since the $g_i$'s and the $h_j$'s generate the same ideal $J$, we can write for all $i$, $g_i=\sum_jr_{ji}h_j$ for some $r_{ji}\in\Oh$ and at least one $r_{ji}$ is invertible in $\Oh$. We construct an homomorphism of graded $\Oh$-algebras $\Phi:S'\rightarrow S$  (preserving degrees) given by $\beta_j\mapsto \sum_ir_{ji}\alpha_i$. This induces a natural morphism $\varphi: U\rightarrow\P^{u-1}$, of $\Oh$-schemes, where $U\subset\P^{q-1}$ is the complement of the indeterminacy locus of $\varphi$, given by the vanishing of $\langle \Phi(\beta_1),\ldots,\Phi(\beta_u)\rangle$ in $\P^{q-1}$. Notice that $\pi^{-1}(\M)$ is not contained in the locus of indeterminacy of $\varphi$ because $\Phi(\beta_j)\not\in\M S$ for at least one $j$. Observe also that the construction of $\varphi$ implies both $\sigma\vert_{U}=\varphi^*(\sigma')$ and $\pi\vert_{U}=\pi'\circ \varphi$. In particular, $\mathscr{T}_{\mathcal{K}}(\sigma)\vert_{U}=\varphi^*(\mathscr{T}_{\mathcal{K}}(\sigma'))$ and $\pi^*(T_{\mathcal{K}}(J))\vert_{U}=\varphi^*\pi'^*(T_{\mathcal{K}}(J))$. According to the definition of $\mathscr{F}$ and $\mathscr{F'}$, we deduce $\mathscr{F}\vert_{U}=\varphi^{*}\mathscr{F}'$.

Assume $\pi^{-1}(\M)\not\subset\mbox{Supp}\mathscr{F}$. As observed above $\mbox{Supp}\mathscr{F}$ is closed in the irreducible $\P^{q-1}$. It follows that there exists some $P\in\pi^{-1}(\M)\cap U$, $P\not\in \mbox{Supp}\mathscr{F}$. Since $U\cap \mbox{Supp}\mathscr{F}=\varphi^{-1}(\mbox{Supp}\mathscr{F}')$, we have $\varphi(P)\not\in\mbox{Supp}\mathscr{F'}$ and $\M S'\subset\Phi^{-1}(\M S)\subset\Phi^{-1}(P)=\varphi(P)$. Therefore $\varphi(P)\in \pi'^{-1}(\M)$, showing that $\pi'^{-1}(\M)\not\subset\mbox{Supp}\mathscr{F'}$.\end{proof}

Now that we have defined $T_{\mathcal{K}}$-fullness and $T_{\mathcal{K}}$-dependence, we present some examples showing that the two notions are independent of each other.

\begin{Example}\label{principal5} If $J\subseteq \Oh$ is a principal ideal then $J$ is $T_{\mathcal{K}}$-dependent. Indeed, in this case we see that $\mbox{Supp}\mathscr{F}\subset\P^0$ is empty, opposite to $\pi^{-1}(\M)$ which is not. It is easy to give examples of principal ideals which are not $T_{\mathcal{K}}$-full.\end{Example}

\begin{Example} Let $J=\langle x^2,y\rangle \subset \C\{x,y\}$. We have seen before (cf. Example \ref{naotfull}) that $J$ fails to be $T_{\mathcal{K}}$-full. Let us show that $J$ is $T_{\mathcal{K}}$-dependent. Here $\pi^*(T_{\mathcal{K}}(J))$ is the sheaf associated to $T_{\mathcal{K}}(J)S=\M S.$ Let $\sigma=x^2\alpha_1+y\alpha_2$. Then the sections $\sigma,\,x\frac{\partial \sigma}{\partial x},x\frac{\partial \sigma}{\partial y},y\frac{\partial \sigma}{\partial x}, y\frac{\partial \sigma}{\partial y}$ generate $\mathscr{T}_{\mathcal{K}}(\sigma)$. We investigate the support of $\mathscr{F}$ in $\P^{1}$. We can compute generators of $\Big(\mathscr{T}_{\mathcal{K}}(\sigma):\pi^*(T_{\mathcal{K}}(J))\Big)$ as $\langle \alpha_1x,\alpha_2\rangle\subset S$. Hence the fiber $\pi^{-1}(\M)$ is not contained in the support of $\mathscr{F}$, showing that $J$ is $T_{\mathcal{K}}$-dependent.\end{Example}

\begin{Example} Let $J=\M^4\subset \C\{x,y\}$. Clearly $J$ is $T_{\mathcal{K}}$-full (cf. Examples \ref{potenciadem}, \ref{cota}). Let us show that $J$ fails to be $T_{\mathcal{K}}$-dependent. Notations being as before, $\pi^*(T_{\mathcal{K}}(J))$ is the sheaf in $\P^4$ associated to $T_{\mathcal{K}}(J)S=J S.$ \,\,\,Let \[\sigma=x^4\alpha_1+x^3y\alpha_2+x^2y^2\alpha_3+xy^3\alpha_4+y^4\alpha_5.\] Then $\sigma,\,x\frac{\partial \sigma}{\partial x},x\frac{\partial \sigma}{\partial y},y\frac{\partial \sigma}{\partial x}, y\frac{\partial \sigma}{\partial y}$ are generators of $\mathscr{T}_{\mathcal{K}}(\sigma)$. We check that the support of $\mathscr{F}$ contains $\pi^{-1}(\M)$. To do this we show, with help of SINGULAR, that $\Big(\mathscr{T}(\sigma):\pi^*(T_{\mathcal{K}}(J))\Big)\subset\M S$. We proceed as follows:

\begin{flushleft}
\texttt{
\noindent > ring r=0,(x,y,a(1..5)),ds;\\
> ideal m=x,y;\\
> ideal M=std(m);\\
> poly s=a(1)*x4+a(2)*x3y+a(3)*x2y2+a(4)*xy3+a(5)*y4;\\
> ideal t=s,m*diff(s,x),m*diff(s,y);\\
> ideal q=quotient(t,m4);\\
> ideal Q=std(q);\\
> Q;\\}
\end{flushleft}
SINGULAR presents a list of generators of $\Big(\mathscr{T}(\sigma):\pi^*(T_{\mathcal{K}}(J))\Big)$; all of them belong to $\M S$ as we readily see typing
\begin{flushleft}
\texttt{
> reduce(Q,M);}
\end{flushleft}
\end{Example}

\section{Main Result}\label{teoremas}

The purpose of this Section is to state and prove our result characterizing moduli ideals, then solving Problem \ref{pergunta}. We also derive Corollary \ref{pratico} which has as a consequence an explicit solution to both the Recognition and the Reconstruction problems mentioned in the Introduction, even for the case of non-isolated hypersurface singularities. 

\begin{thm}\label{main} Let $I\subset\Oh$ be an ideal. Then $I$ is a moduli ideal if and only if $I$ is $T_{\mathcal{K}}$-full and $\Delta_1(I)$ is $T_{\mathcal{K}}$-dependent.
\end{thm}

\begin{proof} Let $\Delta_1(I)=\langle g_1,\ldots,g_q\rangle $ and let $\mathscr{F}$ denote the sheaf $\pi^*(T_{\mathcal{K}}(\Delta_1(I)))/\mathscr{T}_{\mathcal{K}}(\sigma)$ on $\P^{q-1}$, as described before. For any $\lambda=(\lambda_1,\ldots,\lambda_q)\in\C^q\setminus \{0\}$ we will denote by $p_{\lambda}\subseteq \C[\alpha_1,\ldots,\alpha_{q}]$ the homogeneous prime ideal generated by $\{\lambda_k\alpha_{\ell}-\lambda_{\ell}\alpha_k\}_{k,{\ell}}$. If $P\in D_+(\alpha_{\ell})\subset\P^{q-1}$ one checks easily that $\mathscr{T}_{\mathcal{K}}(\sigma)_P=T_{\mathcal{K}}\Big(\sigma/\alpha_{\ell}\Big)S_{(P)}$, where parenthetical notation $-\,_{(P)}$ (here and in what follows) indicates the submodule of degree zero elements of the localization of a graded module at $P\in\P^{q-1}$. Moreover, if $p_{\lambda}S\subseteq P\in D_+(\alpha_{\ell})$ then $\lambda_{\ell}\in\C\setminus 0$ and we can use Remark \ref{aritmetica} to obtain the following Estimate on ideals in the stalk of structural sheaf $\mathcal{O}_{\P^{q-1},P}=S_{(P)}$:

\begin{multline}\label{estimate}
T_{\mathcal{K}}\Big(\sum_k\Big(\lambda_k-\frac{\lambda_{\ell}\alpha_k}{\alpha_{\ell}}\Big)g_k\Big)S_{(P)}\subseteq\sum_k T_{\mathcal{K}}\Big(\Big(\lambda_k-\frac{\lambda_{\ell}\alpha_k}{\alpha_{\ell}}\Big)g_k\Big)S_{(P)}\subseteq \sum_k \Big(\lambda_k-\frac{\lambda_{\ell}\alpha_k}{\alpha_{\ell}}\Big)T_{\mathcal{K}}(g_k)S_{(P)}\subseteq IP_{(P)}
\end{multline}

To prove the ``$\Rightarrow$" assertion in our statement, assume $I$ to be a moduli ideal. We have seen (cf. Proposition \ref{tfullehnecess}) that $I$ is $T_{\mathcal{K}}$-full; hence $\mathscr{F}=\pi^*(I)/\mathscr{T}_{\mathcal{K}}(\sigma)$. Now we verify that $\Delta_1(I)$ is $T_{\mathcal{K}}$-dependent. Since this is clear if $I=\langle 0\rangle$, we assume $I=T_{\mathcal{K}}(f)\neq \langle 0\rangle$, with $f=\sum_k r_kg_k$ for some $r_k\in\Oh$.  We write $r_k=\lambda_k+s_k$, for certain $\lambda_k\in\C$ and $s_k\in\M$, for all $k=1,\ldots,q$. Since $I\neq \langle 0 \rangle$ we can use Remark \ref{easy},\,(v) and Nakayama's Lemma to check $f\not\in\M\Delta_1(I)$. Hence, at least one of the $r_k$'s is an invertible element of $\Oh$. Using Remark \ref{aritmetica} again we may assume $(r_1,\ldots,r_q)=(\lambda_1,\ldots,\lambda_q)\in\C^q\setminus \{0\}$. We consider $P=p_{\lambda}S+\M S\in\P^{q-1}$. Then $P\in\pi^{-1}(\M)$, so that $\pi^*(I)_{P}=I_{\pi(P)}\otimes S_{(P)}=I_{\M}\otimes S_{(P)}=IS_{(P)}$. On the other hand (say if $P\in D_+(\alpha_{\ell})$), our assumption $I=T_{\mathcal{K}}(f)$ together with Estimate (\ref{estimate}) above produces
\begin{multline*}IS_{(P)}=T_{\mathcal{K}}\Big(\sum_k\lambda_k g_k\Big) S_{(P)}\subseteq T_{\mathcal{K}}\Big(\sum_k\Big(\lambda_k-\frac{\lambda_{\ell}\alpha_k}{\alpha_{\ell}}\Big)g_k\Big)S_{(P)}+T_{\mathcal{K}}\Big(\lambda_{\ell}\sigma/\alpha_{\ell}\Big)S_{(P)}\subseteq IP_{(P)}+\mathscr{T}_{\mathcal{K}}(\sigma)_P\subseteq IS_{(P)}.
\end{multline*}
It follows that $IP_{(P)}+\mathscr{T}_{\mathcal{K}}(\sigma)_P=IS_{(P)}$. Now we use Nakayama's Lemma to obtain $\pi^*(I)_{P}=IS_{(P)}=\mathscr{T}_{\mathcal{K}}(\sigma)_P$. Hence $P\not\in \mbox{Supp}\mathscr{F}$, which completes the proof that $\Delta_1(I)$ is $T_{\mathcal{K}}$-dependent.

To prove the ``$\Leftarrow$" assertion in the statement $I$ is assumed to be $T_{\mathcal{K}}$-full and $\Delta_1(I)$ is assumed to be $T_{\mathcal{K}}$-dependent. Then $I=T_{\mathcal{K}}(\Delta_1(I))$ and there exists some $P\in\pi^{-1}(\M)\setminus \mbox{Supp}\mathscr{F}$. It follows that $$\pi^*(I)_{P}=\pi^*(T_{\mathcal{K}}(\Delta_1(I))_{P}=\mathscr{T}_{\mathcal{K}}(\sigma)_P.$$ We make use of the decomposition $S=\C[\alpha_1,\ldots,\alpha_q]+\M S$ to simplify our argument. Indeed, we obtain $P=P\cap\C[\alpha_1,\ldots,\alpha_q]+P\cap\M S=P\cap\C[\alpha_1,\ldots,\alpha_q]+\M S$. Therefore, there exists some $(\lambda_1,\ldots,\lambda_q)\in\C^q\setminus \{0\}$, $p_{\lambda}\supseteq P\cap\C[\alpha_1,\ldots,\alpha_q]$, such that the homogeneous prime ideal $p_{\lambda}S+\M S$ also belongs to $\pi^{-1}(\M)\setminus \mbox{Supp}\mathscr{F}$.

With this simplification - i. e., assuming $P=p_{\lambda}S+\M S\in D_+(\alpha_{\ell})$, as before - the displayed equality of stalks above reads $IS_{(P)}=T_{\mathcal{K}}\Big(\lambda_{\ell}\sigma/\alpha_{\ell}\Big)S_{(P)}.$ We will prove $I=T_{\mathcal{K}}(f)$, for $f=\sum_k\lambda_kg_k\in\Oh$. Estimate (\ref{estimate}) here gives 

\begin{multline*}
IS_{(P)}\subseteq T_{\mathcal{K}}\Big(\sum_k\Big(\lambda_k-\frac{\lambda_{\ell}\alpha_k}{\alpha_{\ell}}\Big)g_k\Big)S_{(P)}+T_{\mathcal{K}}\Big(\sum_k\lambda_k g_k\Big)S_{(P)}\subseteq IP_{(P)}+T_{\mathcal{K}}\Big(\sum_k\lambda_k g_k\Big)S_{(P)}\subseteq IS_{(P)}.
\end{multline*}
We obtain $IP_{(P)}+T_{\mathcal{K}}(\sum_k\lambda_kg_k)S_{(P)}=IS_{(P)}$ and by Nakayama's Lemma again we deduce $T_{\mathcal{K}}(f)S_{(P)}=IS_{(P)}$. 

Since $T_{\mathcal{K}}(f)\subseteq I$, we should verify $I\subseteq T_{\mathcal{K}}(f)$. It suffices to check that if $h\in I$ and there are some $g\in T_{\mathcal{K}}(f)$ and $G_1,G_2\in S$, homogeneous of same degree with $G_2\not\in P$ such that $G_2h=gG_1$, then $h\in T_{\mathcal{K}}(f)$. With this simplification, an argument with divisibility and degree in the unique factorization domain $S$ shows $\xi h=gH$, for some homogeneous element $H\in S$ and some degree zero factor $\xi$ of $G_2$. In particular, $\xi\in\Oh\setminus \M$. Again by degree reasons, $H$ has degree zero, i.e., it belongs to $\Oh$. This shows that $h\in T_{\mathcal{K}}(f)$. Hence $I\subseteq T_{\mathcal{K}}(f)$ and we have shown that $I$ is a moduli ideal of $\Oh$.\end{proof}

We will keep the notation as in the proof of Theorem \ref{main}. Then we have the corollary below. 

\begin{Corollary}\label{pratico}
Let $I\subset \Oh$ be an ideal and let $\Delta_1(I)=\langle g_1,\ldots,g_q\rangle $. Then $I$ is a moduli ideal if and only if $I=T_{\mathcal{K}}(\sum_k\lambda_kg_k)$ for $p_{\lambda}$ varying in a (non empty) Zariski open subset of $\P_{\C}^{q-1}$.
\end{Corollary}

\begin{proof} For the proof of the non-trivial part of the statement, assume that $I$ is a moduli ideal. The canonical inclusion $\C\rightarrow \Oh$ induces $\C[\alpha_1,\ldots,\alpha_{q}]\rightarrow S$, which is a graded inclusion of $\C$-algebras; hence we obtain a proper and dominant morphism of schemes $\rho:\P^{q-1}\rightarrow \P_{\C}^{q-1}$. From the Theorem, we know that $\Delta_1(I)$ is $T_{\mathcal{K}}$-dependent, so that $\pi^{-1}(\M)\setminus \mbox{Supp}\mathscr{F}$ is a non empty Zariski open subset of $\pi^{-1}(\M)$. It follows that its image under $\rho$ is a non-empty open Zariski subset $U\subset\P_{\C}^{q-1}$. Homogeneous prime ideals $p_{\lambda}$ in $U$ correspond bijectively, via $\rho$, to homogeneous prime ideals of type $p_{\lambda}S+\M S\in\pi^{-1}(\M)\setminus \mbox{Supp}\mathscr{F}$. Since the Theorem also guarantees that $I$ is $T_{\mathcal{K}}$-full, we can proceed as in its proof to show that $I=T_{\mathcal{K}}(\sum_k\lambda_kg_k)$, for all $\lambda=(\lambda_1,\ldots,\lambda_q)$ such that $p_{\lambda}\in U$.\end{proof}

\begin{Remark} The practical character and the usefulness of the above Corollary comes from its consonance with the fact outlined in Example \ref{guia}. Precisely, it says that computing the moduli ideal of a sufficiently general $\C$-linear combination of generators of $\Delta_1(I)$ - accessible in practical examples, as we have illustrated in Section \ref{computador} - reveals whether $I\subset \Oh$ is a moduli ideal.\end{Remark}

We illustrate how one can use the software SINGULAR to check whether an ideal $I\subset\Oh$ is a moduli ideal, without knowing a priori the function germ which realizes it as such. We do this with the ideals presented in Examples \ref{duasvariaveis} and \ref{tresvariaveis}, respectively.

\begin{Example}(Example \ref{duasvariaveis}, continued): We have computed the anti-derivatives ideal of $$I=\langle 3xy^2+x^6,y^3,x^5y,x^7\rangle\subset \C\{x,y\}$$ as $\Delta_1(I)=\langle y^3+x^5y, x^4y^2,x^7,x^6y\rangle\subset I$. A SINGULAR check shows that $I$ is $T_{\mathcal{K}}$-full. We consider the corresponding sheaf $\mathscr{F}$ in $\P^3$ and investigate its support. In order to do so we use SINGULAR again:

\begin{flushleft}
\texttt{
\noindent > ring r=0,(x,y,a(1..4)),ds;\\
> ideal i=3xy2+x6,y3,x5y,x7;\\
> poly s=a(1)*(y3+x5y)+a(2)*x4y2+a(3)*x7+a(4)*x6y;\\
> ideal m=x,y;\\
> ideal t=s,m*diff(s,x),m*diff(s,y);\\
> ideal q=quotient(t,i);\\
> ideal Q=std(q);\\
> Q;}
\end{flushleft}

The output is a list with twenty homogeneous generators of $\Big(\mathscr{T}_{\mathcal{K}}(\sigma):\pi^*(I)\Big)$, the sixth element of which being
\begin{multline*}
15\alpha_1^3\alpha_3+52x\alpha_1^2\alpha_3\alpha_4-121y\alpha_1\alpha_2\alpha_3\alpha_4+36y\alpha_1^2\alpha_4^2+15x^2\alpha_1\alpha_3\alpha_4^2+30xy\alpha_1^2\alpha_2^2+28x^2\alpha_1\alpha_3\alpha_4^2+\\24xy\alpha_1\alpha_4^3+42x^3\alpha_1^2\alpha_2\alpha_4+24x^2y\alpha_1\alpha_2^2\alpha_4-54x^4\alpha_2^2\alpha_3\alpha_4+41x^4\alpha_1\alpha_2\alpha_4^2+4x^3y\alpha_2^2\alpha_4^2+14x^5\alpha_2\alpha_4^3.\end{multline*}

Clearly this element does not belong to $\M[\alpha_1,\alpha_2,\alpha_3,\alpha_4]$ due to the presence of the term $15\alpha_1^3\alpha_3$. It follows that $\Delta_1(I)$ is $T_{\mathcal{K}}$-dependent. According to the proof of Theorem, any $\C$-linear combination \[\lambda_1(y^3+x^5y)+\lambda_2x^4y^2+\lambda_3x^7+\lambda_4x^6y\] with both $\lambda_1\neq 0$ and $\lambda_3\neq 0$ represents a function germ $f$, with an isolated singularity, of which $I$ is a moduli ideal.
\end{Example}

\medskip

\begin{Example}(Example \ref{tresvariaveis}, continued): We have computed the anti-derivatives ideal of $$I=\langle x^3,x^2y,x^2z,xy^3z,y^4z,xy^2z^2,y^3z^2,y^2z^3\rangle\subset \C\{x,y,z\}$$ as $\Delta_1(I)=\langle x^3,x^2y^2z,y^3z^2\rangle$. A SINGULAR check shows that $I$ is $T_{\mathcal{K}}$-full. We consider the corresponding sheaf $\mathscr{F}$ in $\P^{2}$ and investigate its support. In order to do this we use SINGULAR again:

\begin{flushleft}
\texttt{
\noindent > ring r=0,(x,y,z,a(1..3)),ds;\\
> ideal i=x3,x2y,x2z,xy3z,y4z,xy2z2,y3z2,y2z3;\\
> poly s=a(1)*(x3)+a(2)*(x2y2z)+a(3)*(y3z2);\\
> ideal m=x,y,z;\\
> ideal t=s,m*diff(s,x),m*diff(s,y),m*diff(s,z);\\
> ideal q=quotient(t,i);\\
> ideal Q=std(q);\\
> Q;}
\end{flushleft}

The output now is a list with ten homogeneous generators of $\Big(\mathscr{T}_{\mathcal{K}}(\sigma):\pi^*(I)\Big)$. Its fifth element is $$3\alpha_1^2\alpha_3^2-xy\alpha_1\alpha_2^2\alpha_3$$ and again, this element does not belong to $\M[\alpha_1,\alpha_2,\alpha_3]$ due to the presence of the term $3\alpha_1^2\alpha_3^2$. Hence $\Delta_1(I)$ is $T_{\mathcal{K}}$-dependent. The proof of Theorem guarantees that any $\C$-linear combination \[\lambda_1x^3+\lambda_2x^2y^2z+\lambda_3y^3z^2\] with both $\lambda_1\neq 0$ and $\lambda_3\neq 0$ represents a function germ $f$, defining germ of a surface with a non-isolated singularity at $0\in\C^3$, of which $I$ is a moduli ideal.
\end{Example}









\bibliographystyle{amsalpha}
\bibliography{Biblio}

\end{document}